\documentclass[12pt]{amsart}
\pdfoutput=1
\usepackage[foot]{amsaddr}

\usepackage[utf8]{inputenc}
\usepackage{indentfirst}
\usepackage{amsmath}
\usepackage{amsthm}
\usepackage{amssymb}
\usepackage{amsfonts}
\usepackage{microtype}
\usepackage{fancyhdr}
\usepackage{tikz-cd}
\usetikzlibrary{arrows}
\usetikzlibrary{positioning}
\usepackage{xcolor}
\usepackage[colorlinks]{hyperref}
\usepackage{mathtools}
\usepackage{thmtools}

\usepackage{import}
\usepackage{xifthen}
\usepackage{pdfpages}
\usepackage{transparent}

\usepackage[
  margin=1in,
  includefoot,
  footskip=30pt,
]{geometry}

\usepackage{csquotes}

\usepackage[style=alphabetic, citestyle=alphabetic]{biblatex}
\hypersetup{
  colorlinks = true,
  citecolor = {purple!70!black},
  linkcolor = {red!60!black}
}

\newcommand{\CC}{\mathbb{C}}
\newcommand{\QQ}{\mathbb{Q}}

\newcommand{\OO}{\mathcal{O}}
\newcommand{\PP}{\mathbb{P}}

\newcommand{\ZZ}{\mathbb{Z}}

\newcommand{\Cone}{\mathrm{Cone}}

\newcommand{\Ext}{\mathrm{Ext}}

\newcommand{\Hom}{\mathrm{Hom}}

\newcommand{\Pic}{\mathrm{Pic}}

\newcommand{\Coh}{\mathrm{Coh}}
\newcommand{\Dbcoh}{D^b_{\!\mathrm{coh}}}

\newcommand{\cA}{\mathcal{A}}

\newcommand{\infinity}{\infty}

\newcommand{\dual}{{\scriptstyle\vee}}
\newcommand{\iso}{\simeq}
\newcommand{\caniso}{\cong}

\newcommand{\monoarrow}{\hookrightarrow}
\newcommand{\epiarrow}{\twoheadrightarrow}

\newcommand{\im}{\operatorname{im}}
\newcommand{\rk}{\operatorname{rk}}

\newcommand{\coker}{\operatorname{coker}}

\newcommand{\supp}{\operatorname{supp}}

\declaretheoremstyle[
headformat=\NUMBER.\,\NAME\NOTE,
postheadspace=.5em,
spaceabove=6pt,
headfont=\normalfont\small\scshape,
notefont=\normalfont\small\mdseries, notebraces={(}{)},
bodyfont=\normalfont\itshape
]{plainswap}
\declaretheoremstyle[
headformat=\NAME\NOTE,
postheadspace=.5em,
spaceabove=6pt,
headfont=\normalfont\small\scshape,
notefont=\normalfont\small\mdseries, notebraces={(}{)},
bodyfont=\normalfont\itshape
]{nonumplainswap}
\declaretheoremstyle[
headformat=\NUMBER.\,\NAME\NOTE,
postheadspace=.5em,
spaceabove=6pt,
headfont=\normalfont\small\scshape,
notefont=\normalfont\mdseries, notebraces={(}{)},
bodyfont=\normalfont
]{definitionswap}
\declaretheoremstyle[
headformat=\NAME\NOTE,
postheadspace=.5em,
spaceabove=6pt,
headfont=\normalfont\itshape,
notefont=\mdseries, notebraces={(}{)},
bodyfont=\normalfont
]{myremark}

\declaretheorem[style=plainswap, name=Theorem, sharenumber=subsection]{theorem}

\declaretheorem[style=plainswap, numberlike=theorem, name=Proposition]{proposition}

\declaretheorem[style=plainswap, numberlike=theorem, name=Lemma]{lemma}
\declaretheorem[style=plainswap, numberlike=theorem, name=Corollary]{corollary}

\theoremstyle{definition}
\declaretheorem[style=definitionswap, numberlike=theorem, name=Definition]{definition}

\declaretheorem[style=definitionswap, numberlike=theorem, name=Example]{example}
\declaretheorem[style=definitionswap, numberlike=theorem, name=Remark]{remark}
\declaretheorem[style=definitionswap, numberlike=theorem, name=Question]{question}

\numberwithin{equation}{theorem}

\usepackage{quiver}

\title{Generators vs. classical generators in derived categories of curves}
\author{Dmitrii Pirozhkov$^{1, 2}$}
\address{$^1$Steklov Mathematical Institute of Russian Academy of Sciences, Moscow, Russia}
\address{$^2$Laboratory of Mirror Symmetry, National Research University Higher School of Economics, Moscow, Russia}
\email{dpirozhkov@mi-ras.ru}

\begin{document}

\begin{abstract}
This is mostly an expository note about an example communicated to the author by Aise Johan de Jong. In a triangulated category $T$ an object $G$ is said to be a classical generator when the smallest triangulated subcategory containing $G$ coincides with the whole $T$, and it is said to be a generator when the orthogonal complement to $G$ in $T$ is zero, i.e., when any non-zero object of $T$ admits a non-zero map from a shift of $G$. Any classical generator is a generator, but not vice versa. We discuss a simple algebro-geometric example of a non-classical generator in the derived category of coherent sheaves on any smooth proper curve of genus $g \geq 2$. We also overview what is known and what is not known, in general, about generators and classical generators on curves.
\end{abstract}
\maketitle

\section{Introduction}

In a triangulated category $T$ there are two different notions of what it means for an object~$G \in T$ to ``generate'' the whole triangulated category:

\begin{definition}
  \label{def:classical_generator}
  Let $T$ be a triangulated category. For an object $G \in T$ denote by~$\langle G \rangle \subset T$ the smallest triangulated subcategory of~$T$ containing~$G$ and closed under taking direct summands. An object $G$ is a \textit{classical generator} of~$T$ if $\langle G \rangle = T$.
\end{definition}

\begin{definition}
  \label{def:generator}
  Let $T$ be a triangulated category. For an object $G \in T$ denote by $G^\perp \subset T$ the subcategory of objects $E \in T$ such that $\Hom^{\bullet}_{T}(G, E) = 0$. An object $G$ is a \textit{generator} of~$T$ if~$G^\perp = 0$.
\end{definition}

\begin{remark}
  \label{rem:generators_can_be_sets}
  Both definitions may be stated for an arbitrary subset $S \subset T$ of objects instead of a single object $G$. A subset $S \subset T$ which is a generator in the sense of Definition~\ref{def:generator} is often called a \textit{spanning class} (see, e.g.,~\cite[Def.~1.47]{huybrechts-fm}).
\end{remark}

\begin{remark}
  We follow the terminology from \cite{bondal-vdbergh}. There are other variants: some sources use the term \emph{split generator} for what we call a classical generator, and \emph{weak generator} for what we call a generator.
\end{remark}

Both these definitions are natural and useful. The notion of a classical generator is closer in spirit to the idea of presenting a basic algebraic object (group, ring, module) in terms of generators and relations, while the condition of being a generator is usually much easier to check and suffices for many purposes. Any classical generator is a generator (see, e.g.,~\cite[Tag~09SL]{stacks-project}), but not necessarily the other way around.

\begin{example}
  \label{ex:local_ring_counterexample}
  A standard counterexample involves a regular local Noetherian ring $R$ with the maximal ideal $\mathfrak{m}$: if $T$ is the bounded derived category of finitely generated modules over $R$, then the simple module~$R / \mathfrak{m}$ is a generator of $T$ by Nakayama's lemma, but the subcategory~$\langle R / \mathfrak{m} \rangle \subset T$ consists only of objects supported at $\mathfrak{m} \in \mathrm{Spec}(R)$ and does not include, for example, $R$ itself.  
\end{example}

The relation between the two definitions is to a certain degree clarified by the following fundamental theorem. For background on compact objects and compactly generated categories see, for example,~\cite{bondal-vdbergh}; we won't use these notions outside of the introduction.

\begin{theorem}[{Ravenel--Neeman~\cite{neeman-localization}}]
  \label{thm:ravenel_neeman}
  Let $T$ be a compactly generated triangulated category, and let $T^c \subset T$ be the subcategory of compact objects. For a compact object $G \in T^c$ the following are equivalent:
  \begin{itemize}
  \item $G$ is a classical generator of $T^c$;
  \item $G$ is a generator of $T$.
  \end{itemize}
\end{theorem}

In the Example~\ref{ex:local_ring_counterexample} above, where $T$ is the bounded derived category of finitely generated modules over $R$, the module $R / \mathfrak{m}$ is not a classical generator because while it is a generator of the ``compact'' bounded derived category of finitely generated $R$-modules, it fails to be a generator of the ambient compactly generated category, namely the unbounded derived category of arbitrary $R$-modules; indeed, the function field of $R$ is an object of this larger triangulated category which lies in the orthogonal to $R / \mathfrak{m}$.

In this note we concentrate on some especially nice triangulated categories appearing in algebraic geometry. Let $X$ be a smooth projective variety over the field $\CC$. The bounded derived category $\Dbcoh(X)$ of coherent sheaves on $X$ is a triangulated category, and from now on we will exclusively consider triangulated categories of that form. Such categories admit explicit classical generators:

\begin{example}[{see, e.g.,~\cite[Tag~0BQT]{stacks-project}}]
  \label{ex:classical_generators_exist}
  Let $j\colon X \monoarrow \PP^N$ be a closed embedding of the projective variety $X$ into some projective space. Then the direct sum
  \[ \OO_X \oplus j^* \OO_{\PP^N}(-1) \oplus \dots \oplus j^* \OO_{\PP^N}(-N) \]
  is a classical generator of $\Dbcoh(X)$.
\end{example}

\begin{remark}
  The existence of classical generators for categories of algebro-geometric origin is known in a much broader generality, the projectivity of $X$ is used only for simplicity. See~\cite{bondal-vdbergh} or more recent work like~\cite{neeman-generators} for more general statements.
\end{remark}

The difference between generators and classical generators is not just a matter of esoteric relevance; if, by some miracle, these two notions were equivalent for categories of the form~$\Dbcoh(X)$, this would have actual useful consequences in algebraic geometry. For example, let $C$ be a smooth projective curve. It is not hard to check that a semistable vector bundle on a curve $C$ can never be a \textit{classical generator} of~$\Dbcoh(C)$ (see Corollary~\ref{cor:semistables_form_triangulated_subcategory}). Somewhat surprisingly, the stronger fact that it can never be a \textit{generator} of $\Dbcoh(C)$ appeared in an unpublished work of Faltings (stated in a different way; see \cite[Lem.~3.1]{seshadri}) and is extremely important for the study of moduli spaces of semistable vector bundles on curves; in fact, a modern proof of the projectivity of these moduli spaces relies on this quite complicated observation~\cite[Ch.~3]{spec-book}. If the theory of triangulated categories somehow helped us upgrade the property of not being a classical generator to the property of not being a generator, even in this very particular case, this would give a quick and conceptual proof of Faltings' result. Unfortunately, this cannot happen: even in triangulated categories as nice as $\Dbcoh(X)$ for a smooth projective variety $X$ a generator does not have to be a classical generator.

However, the standard Example~\ref{ex:local_ring_counterexample} of a generator which is not a classical generator cannot be adapted to the setting where $X$ is assumed to be projective. Nakayama's lemma only applies in the local situation. At best, we can take the set of skyscrapers at all points of~$X$ and say that this set generates $\Dbcoh(X)$ as in Remark~\ref{rem:generators_can_be_sets}, but an uncountable collection of coherent sheaves is not really similar to a single coherent sheaf. Thus, to demonstrate a generator of $\Dbcoh(X)$ that is not a classical generator, we need to find another good reason for an object to have a trivial orthogonal.

When $X$ is a curve, it turns out that the Riemann--Roch theorem can be a pretty good reason. The following example was suggested to me by Aise Johan de Jong in 2018\footnote{\ldots after I spent a month or two trying to prove that when $X$ is a curve this cannot happen.} and exists on any curve of genus at least two.

\begin{theorem}[{de Jong, private communication}]
  \label{thm:main_example}
  Let $C$ be a smooth projective curve of genus at least $2$, and let $L$ be a line bundle of degree $1$ on $C$ with no global sections. Then the direct sum $\OO_C \oplus L$ is a generator of $\Dbcoh(C)$, but not a classical generator.
\end{theorem}

Note that such a line bundle $L$ exists since any line bundle of degree one with a global section is of the form $\OO_{C}(p)$ for some point $p \in C$, while $\dim \Pic^{1}(C) = g > \dim C = 1$.

It appears that in a vast body of research dedicated to the study of derived categories of coherent sheaves there is still no reference with an explicit example of a non-classical generator in a category of the form $\Dbcoh(X)$ with a smooth projective $X$. Thus, in this note we present a full proof of Theorem~\ref{thm:main_example} following the argument by de Jong. The proof assumes some familiarity with basic properties of (semi-)stable vector bundles on curves as well as, of course, with derived categories of coherent sheaves, but is otherwise self-contained. We also include a discussion of some open questions related to generators and classical generators in bounded derived categories of coherent sheaves on curves, which is the only part of this note with any claim to originality.

\textbf{Outline.} Along the way, in Section~\ref{sec:generators}, we prove a necessary and sufficient criterion for an object $E \in \Dbcoh(C)$ to be a generator (Theorem~\ref{thm:generators_in_curves_criterion}). In Section~\ref{sec:classical_generators} we highlight the fact that a collection of simple objects in an abelian category with no non-trivial maps between them can be a classical generator only in very rare circumstances (Lemma~\ref{lem:extensions_do_not_generate_everything}) and use this observation to complete the proof of Main Theorem~\ref{thm:main_example}. In Section~\ref{sec:questions}, divided into three subsections, we discuss further questions one could ask about (classical) generators of the categories $\Dbcoh(C)$. First, in Subsection~\ref{ssec:unstable_classical_generators} we show that on curves of genus one any generator is a classical generator (Corollary~\ref{cor:genus_one_generators}) by proving that certain sufficiently unstable vector bundles are always classical generators (Proposition~\ref{prop:sufficiently_unstable_bundles}). Second, in Subsection~\ref{ssec:behavior_in_families} we ask whether the property of being a classical generator is a closed condition in families of objects. Third, in Subsection~\ref{ssec:slow_generators} we check that most classical generators occurring in this note are somewhat well-behaved in the sense that their generating time has an upper bound depending only on the genus of the curve.

\textbf{Related work.} As mentioned above, figuring out which objects of~$\Dbcoh(C)$ are generators is a question relevant to the study of moduli spaces of various geometric objects on~$C$; the prototypical example is Faltings' construction (see \cite{seshadri} or, for a more modern introduction,~\cite[Ch.~3]{spec-book}) of an ample line bundle on the moduli space of semistable vector bundles on curves. For an overview of these applications see, e.g.,~\cite{popa-theta-series}. As for the study of classical generators on curves, the key result is Orlov's construction~\cite{orlov-curves} of an ``optimal'' classical generator on each smooth projective curve, and the work by Ballard, Favero and Katzarkov~\cite{ballard-favero-katzarkov}, who studied the behavior of some other classical generators. The existence of classical generators and certain properties of them such as the generating time on other classes of varieties (higher dimensional, singular, etc.) is an active research area; see, e.g., \cite{neeman-generators} as a starting point. Additionally, it appears that the example from Theorem~\ref{thm:main_example} is known in some circles, e.g., to Wahei Hara.

\textbf{Acknowledgments.} This text would not exist were it not for Aise Johan de Jong who explained Theorem~\ref{thm:main_example} to me. I thank Alexander Kuznetsov and Dmitri Orlov for helpful discussions about derived categories of curves. I had been vaguely planning to write a thorough discussion of de Jong's example for a long time, and was pushed to action during the school \textit{Autumn ALGEULER 2025} in Saint-Petersburg, when a brief mention of this example during my talk led to an unexpected amount of interest from the audience. 
The study has been funded within the framework of the HSE University Basic Research Program.

\section{Generators on curves}
\label{sec:generators}

Let $C$ be a smooth projective curve over the field $\CC$. Let $\Coh(C)$ be the category of coherent sheaves on $C$. Recall that since $\Coh(C)$ has homological dimension one, any object in the bounded derived category $\Dbcoh(C)$ of coherent sheaves on $C$ is formal:

\begin{lemma}[{see, e.g., \cite[Cor.~3.15]{huybrechts-fm}}]
  \label{lem:objects_on_curves_split}
  Any object $E \in \Dbcoh(C)$ is isomorphic to the direct sum $\oplus_{i \in \ZZ} \mathcal{H}^{i}(E)[-i]$.
\end{lemma}

Thus we can harmlessly transfer some terminology used for coherent sheaves to arbitrary objects in $\Dbcoh(C)$:

\begin{definition}
  An object $E \in \Dbcoh(C)$ is said to be:
  \begin{itemize}
  \item a \textit{torsion object} if each cohomology sheaf of $E$ is a torsion coherent sheaf;
  \item a \textit{locally free object} if each cohomology sheaf of $E$ is locally free;
  \item a \textit{semistable object of slope $\lambda \in \QQ$} if all cohomology sheaves of $E$ are semistable vector bundles of slope $\lambda$;
  \item a \textit{semistable object} if it is semistable for some $\lambda \in \QQ$ or if it is a torsion object (``semistable of slope $\infinity$'').
  \end{itemize}
\end{definition}

It turns out that there is a simple criterion for describing the generators in the triangulated category~$\Dbcoh(C)$:

\begin{theorem}
  \label{thm:generators_in_curves_criterion}
  An object $E \in \Dbcoh(C)$ is a generator of $\Dbcoh(C)$ if and only if it is not semistable.
\end{theorem}

Below we give a detailed proof for the ``easy'' implication: namely, in Corollary~\ref{cor:unstable_implies_generator} we show that any non-semistable object is a generator of $\Dbcoh(C)$, which essentially follows from the Riemann--Roch formula. This is the implication we need to prove Theorem~\ref{thm:main_example}. The converse implication, sometimes called ``Faltings' lemma'', is, at its core, a deep fact about deformation theory of vector bundles on curves. We discuss it briefly at the end of the section.

Recall a basic fact about coherent sheaves on a curve. We leave the proof as an exercise to the reader.

\begin{lemma}
  \label{lem:coherent_sheaves_on_curves}
  Let $\mathcal{F} \in \Coh(C)$ be a coherent sheaf. Then $\mathcal{F}$ splits into a direct sum $\mathcal{F} \iso T \oplus \mathcal{E}$, where $T$ is a torsion coherent sheaf and $\mathcal{E}$ is a locally free sheaf.
\end{lemma}

For convenience of reference we also formulate a trivial consequence of Lemma~\ref{lem:objects_on_curves_split}:

\begin{lemma}
  \label{lem:kernel_cokernel_in_cone}
  Let $T \subset \Dbcoh(C)$ be a triangulated subcategory closed under direct summands, and let $F, G$ be two coherent sheaves on $C$. If $F$ and $G$ lie in the subcategory $T$, then for any morphism $\phi\colon F \to G$ both kernel and cokernel of $\phi$ are objects in $T$.
\end{lemma}

\begin{proof}
  Since $T$ is a triangulated subcategory, the cone $\Cone(\phi) \in \Dbcoh(C)$ lies in~$T$. By Lemma~\ref{lem:objects_on_curves_split} the cone is isomorphic to a direct sum of $\ker(\phi)[1]$ and $\coker(\phi)$. By assumption the subcategory~$T$ is closed under passing to direct summands, so both $\ker(\phi)$ and $\coker(\phi)$ belong to $T$.
\end{proof}

We will use the Riemann--Roch formula in the following form. It can be easily deduced from the Hirzebruch--Riemann--Roch theorem. Here the notation $\mu(E)$ refers to the \emph{slope} of the bundle $E$, defined as the quotient $\deg(E) / \rk(E)$.

\begin{theorem}
  \label{thm:riemann_roch}
  Let $E$ and $F$ be two non-zero vector bundles on $C$. Then 
  \[ \chi(E, F) = \rk(E) \rk(F) ( 1 - g - \mu(E) + \mu(F)). \]
  In particular, if $\chi(E, F) = 0$, then $\mu(F) = \mu(E) + g - 1$.
\end{theorem}

One of many unique features of curves is that it is very, very hard for two objects in $\Dbcoh(C)$ to be semi-orthogonal. In particular, the following property holds, which is the ``easy'' part of Theorem~\ref{thm:generators_in_curves_criterion}.

\begin{lemma}[{\cite[Lem.~8.3]{seshadri}}]
  \label{lem:semiorthogonality_implies_semistability}
  Let $E, F \in \Dbcoh(C)$ be two nonzero objects. If they are semi-orthogonal, i.e., if~$\Ext^{\bullet}_{\Dbcoh(C)}(E, F) = 0$, then both $E$ and $F$ are semistable objects.
\end{lemma}

\begin{proof}
  Since both $E$ and $F$ are finite direct sums of shifts of coherent sheaves, the semi-orthogonality implies that
  \[ \Ext^{\bullet}_{\Dbcoh(C)}(\oplus_{i \in \ZZ} \mathcal{H}^{i}(E), \oplus_{i \in \ZZ} \mathcal{H}^{i}(F)) = 0, \]
  thus by replacing $E$ with $\oplus \mathcal{H}^{i}(E)$ (and similarly for $F$) we may assume that $E$ and $F$ are both coherent sheaves and both $\Hom(E, F)$ and $\Ext^1(E, F)$ vanish.

  By Lemma~\ref{lem:coherent_sheaves_on_curves} we can split $E \iso E_T \oplus E_V$ and $F \iso F_T \oplus F_V$, where $E_T$ and $F_T$ are torsion coherent sheaves and $E_V$ and $F_V$ are vector bundles on $C$. Note that any nonzero vector bundle admits a nonzero map into any nonzero torsion coherent sheaf. Thus the vanishing of
  \[ \Hom(E_V, F_T) \subset \Hom(E, F) = 0 \]
  implies that at most one of the sheaves $E_V$ and $F_T$ can be non-zero. Similarly, by Serre duality we see that
  \[ \Hom(F_V, E_T \otimes \omega_C)^\dual \caniso \Ext^1(E_T, F_V) \subset \Ext^1(E, F) = 0, \]
  so at most one of the sheaves $E_T$ and $F_V$ can be non-zero. Thus either we have two torsion sheaves $E = E_T$ and $F = F_T$ and we are done with the proof, or both $E$ and $F$ are in fact vector bundles.

  Assume that $E$ is a vector bundle which is not semistable. Then there exists a destabilizing short exact sequence
  \[ 0 \to E^\prime \to E \to E^{\prime \prime} \to 0, \]
  with the strict inequality $\mu(E^{\prime \prime}) < \mu(E)$ of slopes. By the Riemann--Roch formula (Theorem~\ref{thm:riemann_roch}) this implies that
  \[ \chi(E^{\prime \prime}, F) > \chi(E, F) = \dim \Hom(E, F) - \dim \Ext^1(E, F) = 0. \]
  In particular, since the Euler characteristic $\chi(E^{\prime \prime}, F)$ is a positive number, there exists a nonzero morphism $f\colon E^{\prime \prime} \to F$. But then the composition
  \[ E \to E^{\prime \prime} \xrightarrow{f} F \]
  is also nonzero since $E \to E^{\prime \prime}$ is an epimorphism. This is a contradiction, so $E$ is necessarily semistable.

  The semistability of $F$ can be proved in a similar way, but it also follows formally from Serre duality: the vanishing $\Ext^{\bullet}(E, F) = 0$ is equivalent to the vanishing of $\Ext^{\bullet}(F, E \otimes \omega_C)$, so applying the argument above to the pair of semi-orthogonal vector bundles $F$ and $E \otimes \omega_C$ completes the proof.
\end{proof}

\begin{remark}
  The Riemann--Roch argument in the proof also shows that if $E$ is a semistable object with slope $\lambda \in \QQ$ then $F$ is semistable of slope $\lambda + g - 1$.
\end{remark}

As an obvious consequence we get the following:

\begin{corollary}
  \label{cor:unstable_implies_generator}
  Let $E \in \Dbcoh(C)$ be a non-semistable object. Then $E$ is a generator of the category~$\Dbcoh(C)$.
\end{corollary}

\begin{proof}
  By definition we need to check that for any nonzero $F \in \Dbcoh(C)$ the graded vector space~$\Ext^{\bullet}(E, F)$ is not zero. Since $E$ is not semistable, this holds by Lemma~\ref{lem:semiorthogonality_implies_semistability}.
\end{proof}

As we mentioned in Theorem~\ref{thm:generators_in_curves_criterion}, an object $E \in \Dbcoh(C)$ is a generator if and only if it is not semistable. Corollary~\ref{cor:unstable_implies_generator} is the easy direction, but the other direction, proving that a semistable object has a nonzero orthogonal in $\Dbcoh(C)$, is incomparably more difficult. This statement was proved by Faltings and it has deep implications for moduli spaces of semistable vector bundles on curves:

\begin{theorem}[{Faltings, unpublished; see~\cite[Lem.~3.1]{seshadri} or \cite[Prop.~3.5.2]{spec-book}}]
  \label{thm:faltings_orthogonality}
  Let $E \in \Dbcoh(C)$ be a semistable object. Then there exists an object $F \in \Dbcoh(C)$ such that~$\Ext^{\bullet}(E, F) = 0$ and $F \neq 0$.
\end{theorem}

Of course, the original statement was not about objects in the derived category, but only about stable vector bundles. We sketch the reduction to this special case:

\begin{proof}[{Proof (sketch)}]
  It is enough to find $F$ which is semi-orthogonal to each cohomology sheaf of~$E$, so again by replacing~$E$ with $\oplus_{i \in \ZZ} \mathcal{H}^{i}(E)$ we may assume that $E$ is a semistable coherent sheaf. If~$E$ is a torsion sheaf, we can take~$F$ to be the skyscraper sheaf at any point in~$C \backslash \supp(E)$, so from now on let us assume that $E$ is a semistable vector bundle.
  
  By definition $E$ admits a Jordan--Hölder filtration whose associated graded factors, denote them by~$E_1, \dots, E_m$, are all stable vector bundles of slope $\mu := \mu(E)$. By Faltings' result~\cite[Lem.~3.1]{seshadri} for each stable bundle $E_i$ there exists a nonzero vector bundle $F_i$ such that~$\Ext^{\bullet}(E_i, F_i) = 0$. By Lemma~\ref{lem:semiorthogonality_implies_semistability} and Riemann--Roch we know that each $F_i$ is a semistable bundle with slope~$\mu(E) + g - 1$. Then, replacing if necessary each $F_i$ with a direct sum of several copies of $F_i$, we may assume that all bundles $F_i$ have the same rank $r$ and the same degree $d$.

  Now consider the moduli space $\mathcal{M}_{r, d}$ of semistable vector bundles on $C$ with rank $r$ and degree $d$. For each stable factor $E_i$ of $E$ we know that there exists at least one point~$\{ F_i \} \in \mathcal{M}_{r, d}$ of this moduli space such that
  \[ \Hom(E_i, F_i) = 0 \quad \text{and} \quad \Ext^1(E_i, F_i) = 0. \]
  By the upper semicontinuity of cohomology (see, e.g.,~\cite[Thm.~12.8]{hartshorne-ag}) the same vanishing holds for any point of $\mathcal{M}_{r, d}$ in a Zariski-open neighborhood of $\{ F_i \}$. Since $\mathcal{M}_{r, d}$ is known to be irreducible, an intersection of finitely many nonempty open subsets is non-empty, which means that a general semistable vector bundle of rank $r$ and degree $d$ is semi-orthogonal to each stable factor $E_i$ of $E$. Since $E$ is an iterated extension of its stable factors, this implies that a general semistable vector bundle $F \in \mathcal{M}_{r, d}$ of rank $r$ and degree $d$ satisfies~$\Ext^{\bullet}(E, F) = 0$.
\end{proof}

\section{Classical generators on curves}
\label{sec:classical_generators}

By Theorem~\ref{thm:generators_in_curves_criterion} a semistable object $E \in \Dbcoh(C)$ is not a generator, which implies that~$E$ is certainly not a classical generator. This weaker statement is, in fact, not hard to prove directly. For example, if $E$ is a torsion object, then the full subcategory consisting of torsion objects in $\Dbcoh(C)$ is clearly triangulated (exercise: prove this), contains $E$, and is not the whole category, so $\langle E \rangle$ cannot be equal to $\Dbcoh(C)$. For a semistable object $E$ of finite slope~$\lambda \in \QQ$ we will follow a similar strategy: below, in Corollary~\ref{cor:semistables_form_triangulated_subcategory}, we construct a triangulated subcategory of $\Dbcoh(C)$ that contains $E$ and is clearly not equal to $\Dbcoh(C)$. We start with a lemma which will also be used for the proof of Theorem~\ref{thm:main_example}.

\begin{lemma}
  \label{lem:serre_subcategory_implies_triangulated}
  Let $\cA$ be an abelian category, and let $\mathcal{S} \subset \cA$ be a full subcategory satisfying two properties:
  \begin{itemize}
  \item if $\phi\colon A \to A^\prime$ is a morphism between two objects in $\mathcal{S}$, then both $\ker(\phi)$ and $\mathrm{coker}(\phi)$ lie in $\mathcal{S}$;
  \item for any short exact sequence $0 \to A \to B \to A^\prime \to 0$ with $A, A^\prime \in \mathcal{S}$ the middle object $B$ also lies in $\mathcal{S}$.
  \end{itemize}
  Then the subcategory $D^{b}_{\mathcal{S}}(\cA) \subset D^{b}(\cA)$ of objects $E \in D^{b}(\cA)$ whose cohomology sheaves belong to $\mathcal{S}$ is a triangulated subcategory. In particular, an object of $\mathcal{S}$ cannot be a classical generator of $D^{b}(\cA)$ unless $\mathcal{S} = \cA$.
\end{lemma}

\begin{remark}
  A subcategory $\mathcal{S} \subset \cA$ with those two properties is called a \textit{weak Serre subcategory} in~\cite[Tag~02MN]{stacks-project}.
\end{remark}

\begin{proof}
  The subcategory $D^{b}_{\mathcal{S}}(\cA)$ is clearly closed under shifts, so it remains to check that for any morphism $f\colon E \to E^\prime$ between objects $E, E^\prime \in D^{b}_{\mathcal{S}}(\cA)$ the cone $\Cone(f) \in D^{b}(\cA)$ also belongs to $D^{b}_{\mathcal{S}}(\cA)$. Consider a fragment of the long exact sequence of cohomology sheaves for the object~$\Cone(f)$:
  \[ \dots \to \mathcal{H}^{i}(E) \xrightarrow{\mathcal{H}^{i}(f)} \mathcal{H}^{i}(E^\prime) \to \mathcal{H}^{i}(\Cone(f)) \to \mathcal{H}^{i+1}(E) \xrightarrow{\mathcal{H}^{i+1}(f)} \mathcal{H}^{i+1}(E^\prime) \to \dots \]
  All sheaves except the middle one lie in the subcategory $\mathcal{S}$ by assumption. Thus the short exact sequence
  \[ 0 \to \mathrm{coker}(\mathcal{H}^{i}(f)) \to \mathcal{H}^{i}(\Cone(f)) \to \ker(\mathcal{H}^{i+1}(f)) \to 0, \]
  shows that $\mathcal{H}^{i}(\Cone(f))$ is also an object in the subcategory $\mathcal{S}$. This works for any $i \in \ZZ$, so by definition $\Cone(f)$ belongs to $D^{b}_{\mathcal{S}}(\cA)$.
\end{proof}

\begin{corollary}
  \label{cor:semistables_form_triangulated_subcategory}
  Let $\lambda \in \QQ$ be a rational number. Then the full subcategory $D_\lambda \subset \Dbcoh(C)$ of semistable objects of slope $\lambda$ is a triangulated subcategory. In particular, a semistable object of slope $\lambda$ is not a classical generator of $\Dbcoh(C)$.
\end{corollary}

\begin{proof}
  By Lemma~\ref{lem:serre_subcategory_implies_triangulated} it is enough to check two properties:
  \begin{itemize}
  \item if $\phi\colon \mathcal{E} \to \mathcal{E}^\prime$ is a morphism between semistable vector bundles of slope $\lambda$, then both $\ker(\phi)$ and $\mathrm{coker}(\phi)$ are semistable vector bundles of slope $\lambda$;
  \item an extension of two semistable vector bundles of slope $\lambda$ is also semistable of slope $\lambda$.
  \end{itemize}
  Both properties are standard facts about semistable sheaves on curves.
\end{proof}

\begin{remark}
  For any semistable object $E \in \Dbcoh(C)$ of finite slope $\lambda \in \QQ$ the subcategory~$\langle E \rangle \subset \Dbcoh(C)$ classically generated by $E$ is usually smaller than $D_\lambda$. For example, if $C$ is an elliptic curve and $E$ is the trivial line bundle, it is easy to see that $\langle E \rangle$ does not contain any other line bundle of degree zero.
\end{remark}

Some non-semistable objects generate everything. For example, a direct sum of a torsion sheaf with a non-trivial vector bundle is always a classical generator:

\begin{lemma}
  \label{lem:torsion_plus_bundle_generates_everything}
  Let $E \in \Dbcoh(C)$. If $E$ is neither a torsion object nor a locally free object, then~$E$ is a classical generator of $\Dbcoh(C)$.
\end{lemma}

\begin{proof}
  Since any $E \in \Dbcoh(C)$ is isomorphic to the direct sum of shifted coherent sheaves, we may assume that $E$ is a single coherent sheaf. By Lemma~\ref{lem:coherent_sheaves_on_curves} the sheaf $E$ splits into a direct sum $T \oplus \mathcal{E}$, where $T$ is a non-zero torsion sheaf and $\mathcal{E}$ is a non-zero vector bundle on $C$.
  
  Let $p \in C$ be a point in the support of $T$. First we claim that the subcategory $\langle T \rangle$ contains the skyscraper sheaf $\OO_p$. Since $C$ is a curve, this can be shown by appealing to the classification of finitely generated modules over a PID, but here we present an argument which can be adapted to work on an arbitrary variety as long as the support of $T$ is zero-dimensional. Consider the short exact sequence
    \begin{equation}
      \label{eq:skyscraper_presentation}
      0 \to \OO_{C}(-p) \to \OO_C \to \OO_p \to 0.
    \end{equation}
  Derived tensor product with $T$ results in the distinguished triangle
  \[ T \otimes \OO_{C}(-p) \to T \to T \otimes^{L} \OO_p \to T \otimes \OO_{C}(-p) [1] \]
  in the derived category $\Dbcoh(C)$. Since $T$ is a torsion sheaf, the line bundle $\OO_{C}(-p)$ is trivial on a neighborhood of the support of $T$, so $T \otimes \OO_{C}(-p) \iso T$. We conclude that $T \otimes^{L} \OO_{p}$ lies in the subcategory $\langle T \rangle$. Note that this object is not zero since its zeroth cohomology sheaf is the underived tensor product $T \otimes \OO_p$ which is not zero under the assumption that $p$ lies in the support of $T$ (see, e.g.,~\cite[Ex.~3.30]{huybrechts-fm}). Denote by $i\colon \{ p \} \monoarrow C$ the inclusion of the point. Then by the projection formula (see, e.g.,~\cite[(3.11)]{huybrechts-fm}) we have
  \[ T \otimes^{L} i_*(\OO_{p}) \iso R i_* (L i^* T \otimes^{L} \OO_p), \]
  where the object $L i^* T \otimes^{L} \OO_p$ is an object in $\Dbcoh(\{p\})$. Any object in the derived category of sheaves on a point is a direct sum of shifts of $\OO_p$. Thus $T \otimes^{L} i_* \OO_p$ is a direct sum of shifts of several copies of the skyscraper sheaf $\OO_p$. By definition $\langle T \rangle$ is closed under shifts and direct summands, so we have confirmed that $\OO_p$ lies in $\langle T \rangle$.

  Thus it is enough to prove that the direct sum $\OO_p \oplus \mathcal{E}$ is a classical generator of $\Dbcoh(C)$ for any non-zero vector bundle $\mathcal{E}$. If $\mathcal{E} \iso \OO_C$, from the short exact sequence~(\ref{eq:skyscraper_presentation}) we see that $\OO_{C}(-p) \in \langle \OO_p \oplus \OO_C \rangle$. By considering the twists of that short exact sequence we can inductively show that the line bundles $\OO_{C}(-2 p), \OO_{C}(-3 p), \dots$ all lie in $\langle \OO_p \oplus \OO_C \rangle$, and we conclude by Example~\ref{ex:classical_generators_exist} since for some $n > 0$ the line bundle $\OO_{C}(n p)$ is very ample and the sequence above includes arbitrarily many powers of the anti-very ample line bundle.
  
  The case of an arbitrary vector bundle $\mathcal{E}$ follows from this special case: consider the dual vector bundle~$\mathcal{E}^\dual$. Since $\OO_p \oplus \OO_C$ is a classical generator, we know that
  \[ \mathcal{E}^\dual \in \langle \OO_p \oplus \OO_C \rangle \subset \Dbcoh(C). \]
  Taking the (derived) tensor product with $\mathcal{E}$ implies that
  \[ \mathcal{E} \otimes \mathcal{E}^\dual \in \langle (\mathcal{E} \otimes \OO_p) \oplus \mathcal{E} \rangle \subset \Dbcoh(C). \]
  Note that $\mathcal{E} \otimes \OO_p$ is isomorphic to the direct sum of $\rk(\mathcal{E})$ copies of the skyscraper sheaf $\OO_p$, so we have in fact proved that
    \begin{equation}
      \label{eq:endomorphisms_generated_by_bundle}
      \mathcal{E} \otimes \mathcal{E}^\dual \in \langle \OO_p \oplus \mathcal{E} \rangle.
    \end{equation}
  The identity map $\mathcal{E} \to \mathcal{E}$ corresponds to the morphism $\iota\colon \OO_{C} \to \mathcal{E} \otimes \mathcal{E}^\dual$, and the trace map~$\mathcal{E} \otimes \mathcal{E}^\dual \to \OO_C$ provides a splitting for $\iota$ since we always assume characteristic zero. Thus the bundle~$\mathcal{E} \otimes \mathcal{E}^\dual$ has a direct summand isomorphic to $\OO_C$. In particular,~(\ref{eq:endomorphisms_generated_by_bundle}) implies that $\OO_C$ lies in $\langle \OO_p \oplus \mathcal{E} \rangle$, and then
  \[ \langle \OO_p \oplus \mathcal{E} \rangle \supset \langle \OO_p \oplus \OO_C \rangle = \Dbcoh(C), \]
  which means that $\OO_p \oplus \mathcal{E}$ is a classical generator of $\Dbcoh(C)$.
\end{proof}

I don't know a necessary and sufficient criterion for an object in $\Dbcoh(C)$ to be a classical generator. The objects in Lemma~\ref{lem:torsion_plus_bundle_generates_everything} are in some sense ``maximally unstable'', composed of some semistable factors of finite slope and a factor with slope $+\infinity$. Some sufficient instability conditions are discussed later in Corollary~\ref{cor:sufficient_instability_criterion}, but it would be interesting to find a more precise answer.

To show that a non-semistable object can still fail to classically generate $\Dbcoh(C)$ (and prove Theorem~\ref{thm:main_example}) we need a lemma about a set of objects in an abelian category with no interesting morphisms between them. This is a well-known statement (see, e.g., \cite[(1.2)]{ringel-extensions}), but we include the proof for completeness.

\begin{lemma}
  \label{lem:extensions_do_not_generate_everything}
  Let $\cA$ be an abelian category and let $\mathcal{F} \subset \cA$ be a finite set of objects in $\cA$ such that:
  \begin{itemize}
  \item each $F \in \mathcal{F}$ is simple, i.e., $\Hom_{\cA}(F, F) = \CC \cdot \mathrm{id}_{F}$;
  \item for any two non-isomorphic objects $F, F^\prime \in \mathcal{F}$ we have $\Hom_{\cA}(F, F^\prime) = 0$.
  \end{itemize}
  Then the direct sum $\bigoplus_{F \in \mathcal{F}} F$ is a classical generator of the bounded derived category $D^{b}(\cA)$ if and only if any object in $\cA$ is a finite iterated extension of objects in $\mathcal{F}$.
\end{lemma}

\begin{proof}
  Denote by $\mathcal{I} \subset \cA$ the subcategory of objects in $\cA$ which are finite iterated extensions of objects in $\mathcal{F}$, i.e., each $A \in \mathcal{I}$ admits a filtration
    \begin{equation}
      \label{eq:iterated_extension_filtration}
      0 = A_0 \subset A_1 \subset \dots \subset A_r = A
    \end{equation}
  such that the associated graded factors $A_{i+1} / A_i$ are objects in the set $\mathcal{F}$. To prove the statement it is enough to show that the subcategory $\mathcal{I} \subset \cA$ satisfies the two properties from Lemma~\ref{lem:serre_subcategory_implies_triangulated} since the category $D^{b}_{\mathcal{I}}(\cA)$ of complexes with cohomology objects in $\mathcal{I}$ coincides with $D^{b}(\cA)$ if and only if $\mathcal{I} = \cA$. Since the category $\mathcal{I}$ is by definition closed under extensions, we only need to check that for any morphism~$\phi\colon A \to A^\prime$ between objects in $\mathcal{I}$ we have~$\ker(\phi) \in \mathcal{I}$ and~$\mathrm{coker}(\phi) \in \mathcal{I}$. We will prove that the cokernel of~$\phi$ lies in $\mathcal{I}$ and leave the argument for $\ker(\phi)$ to the diligent reader.

  First, let us prove an auxiliary claim. Let $A \in \mathcal{I}$ be an object with a filtration~(\ref{eq:iterated_extension_filtration}) of length $r$. We claim that any nonzero map $f\colon F \to A$ from an object $F \in \mathcal{F}$ is a monomorphism, the quotient lies in $\mathcal{I}$ and admits a filtration~(\ref{eq:iterated_extension_filtration}) of length $r-1$. To prove this, consider the quotient map $g\colon A \epiarrow A_{r} / A_{r-1} \iso F^\prime$ with $F^\prime \in \mathcal{F}$. The composition
  \[ F \xrightarrow{f} A \xrightarrow{g} F^\prime \]
  is a map between two objects in $\mathcal{F}$. By the definition of $\mathcal{F}$ any such map is either an isomorphism or a zero map. If this is an isomorphism, then the map $f\colon F \to A$ splits, so we have $A \iso F \oplus A_{r-1}$ and $\coker(f) \iso A_{r-1}$, as we wanted to show. If the composition~$g \circ f\colon F \to F^\prime = A / A_{r-1}$ is the zero map, then $f\colon F \to A$ factors through a nonzero map~$\overline{f}\colon F \to A_{r-1}$, which by induction on $r$ is a monomorphism with $\coker(\overline{f}) \in \mathcal{I}$ admitting a filtration of length $r-2$. From the short exact sequence
  \[ 0 \to \coker(\overline{f}) \to \coker(f) \to F^\prime \to 0, \]
  we see that $\coker(f)$ lies in $\mathcal{I}$ and admits a filtration of length $r-1$.

  Now consider a morphism $\phi\colon A \to A^\prime$ between objects in $\mathcal{I}$ with filtrations~(\ref{eq:iterated_extension_filtration}) of length~$r$ and~$r^\prime$, respectively. We proceed by induction on the length $r$. If $r = 0$, then $A = 0$ and the statement is trivially true. Otherwise, consider the subobject $f\colon F = A_1 \monoarrow A$ with~$F \in \mathcal{F}$. If the composition $\phi \circ f\colon F \to A^\prime$ is zero, then $\coker(\phi) = \coker(A / A_1 \to A^\prime)$ and the quotient $A / A_1$ has induced filtration of length $r-1$, so this case is covered by the induction hypothesis. If the composition $\phi \circ f$ is not zero, then by assumption and the auxiliary claim above we know that both $f$ and $\phi \circ f$ are monomorphisms, the quotients $A / F$ and $A^\prime / \phi(F)$ both lie in $\mathcal{I}$ and have filtrations~(\ref{eq:iterated_extension_filtration}) of lengths $r-1$ and $r^\prime - 1$, respectively. Since $\phi$ induces an isomorphism on the subobject $F \subset A$, we have an isomorphism of cokernels
  \[ \coker(A \xrightarrow{\phi} A^\prime) \caniso \coker(A / F \to A^{\prime} / \phi(F)), \]
  so we are again done by the induction hypothesis.
\end{proof}

\begin{remark}
  \label{rem:simple_is_not_generator}
  In particular, a single simple object $F \in \cA$ is almost never a classical generator of the derived category~$D^{b}(\cA)$.
\end{remark}

Now we can easily deduce the main result of this note:

\begin{proof}[{Proof of Theorem~\ref{thm:main_example}}]
  Recall that we have a curve $C$ of genus at least $2$ and a line bundle~$L$ of degree $1$ on it with no global sections. By Corollary~\ref{cor:unstable_implies_generator} the direct sum $\OO_C \oplus L$ is a generator of $\Dbcoh(C)$ since it is not semistable, and it remains to prove that $\OO_C \oplus L$ is not a classical generator.
  
  Since $\OO_C$ and $L$ are both line bundles, they have only scalar endomorphisms. By assumption we have $\Hom(\OO_C, L) = H^0(L) = 0$ and also the space $\Hom(L, \OO_C)$ vanishes since $L$ has a strictly positive degree. Thus the set $\{ \OO_C, L \} \subset \Coh(C)$ satisfies the conditions in Lemma~\ref{lem:extensions_do_not_generate_everything}. We conclude that the subcategory $\langle \OO_C \oplus L \rangle \subset \Dbcoh(C)$ consists of objects whose cohomology sheaves are iterated extensions of $\OO_C$ and $L$. Clearly a skyscraper sheaf at any point of $C$ is not such an object. Thus $\OO_C \oplus L$ is not a classical generator of $\Dbcoh(C)$.
\end{proof}

\begin{remark}
  The proof works for line bundles $L$ of any degree between $1$ and $g-1$ with no global sections. However, as soon as $\deg(L) \geq g$, we necessarily have a non-zero global section~$\OO_C \to L$ whose cokernel is a non-zero torsion sheaf $T$, and then $\OO_C$ and $T$ together classically generate $\Dbcoh(C)$ by Lemma~\ref{lem:torsion_plus_bundle_generates_everything}.
\end{remark}

\section{Loose ends and open questions}
\label{sec:questions}

\subsection{Sufficiently unstable bundles are classical generators}
\label{ssec:unstable_classical_generators}

The example in Theorem~\ref{thm:main_example} only exists on curves of genus at least two. This is not an artifact of the proof: in fact, for curves of genus zero (i.e., $\PP^1$) and genus one (elliptic curves) any generator of the derived category of coherent sheaves happens to be a classical generator as well. For $\PP^1$ this is an easy consequence of Lemma~\ref{lem:coherent_sheaves_on_curves} and the fact that any vector bundle on $\PP^1$ is isomorphic to a direct sum of line bundles $\OO_{\PP^1}(i)$; we leave the details to the reader. The genus one case is more interesting and we explain the proof in Corollary~\ref{cor:genus_one_generators} below. The core of the proof is Proposition~\ref{prop:sufficiently_unstable_bundles}, an analogue of Lemma~\ref{lem:torsion_plus_bundle_generates_everything} where instead of a vector bundle and a torsion sheaf we have two vector bundles with a sufficiently strong slope inequality. It just happens that for curves of genus one the inequality becomes very simple.

To state the inequality, let us first recall some standard notation. For a vector bundle $\mathcal{E}$ on a smooth projective curve $C$ define $\mu_{\max}(\mathcal{E})$ to be the largest slope among all non-zero subsheaves of $\mathcal{E}$, and $\mu_{\min}(\mathcal{E})$ to be the smallest slope among all non-zero quotient sheaves of~$\mathcal{E}$. It is well-known that $\mu_{\max}(\mathcal{E})$ and $\mu_{\min}(\mathcal{E})$ are the slopes of the first and the last factor in the Harder--Narasimhan filtration on $\mathcal{E}$. Clearly we have $\mu_{\min}(\mathcal{E}) \leq \mu(\mathcal{E}) \leq \mu_{\max}(\mathcal{E})$.

\begin{proposition}
  \label{prop:sufficiently_unstable_bundles}
  Let $C$ be a smooth projective curve of genus $g \geq 1$. Let $F$ and $G$ be two vector bundles on $C$ such that $\mu_{\max}(F) + g - 1 < \mu_{\min}(G)$. Then $F \oplus G$ is a classical generator of the category~$\Dbcoh(C)$.
\end{proposition}

\begin{proof}
  We proceed by induction on the sum of ranks $\rk(F) + \rk(G)$. Observe that the Riemann--Roch theorem (Theorem~\ref{thm:riemann_roch}) implies that
  \[ \begin{aligned} \chi(F, G) & = \rk(F) \rk(G) (1 - g - \mu(F) + \mu(G)) \geq \\
  & \geq \rk(F) \rk(G) (1 - g - \mu_{\max}(F) + \mu_{\min}(G)) > 0. \end{aligned} \]
  Thus $\Hom(F, G) \neq 0$ and there exists a non-zero morphism $\phi\colon F \to G$.
  
  If both $F$ and $G$ are line bundles, any non-zero map $\phi\colon F \to G$ is a monomorphism with torsion cokernel. Since the slopes of $F$ and $G$ are distinct, $\phi$ cannot be an isomorphism, thus in this case $\coker(\phi)$ is a non-zero torsion sheaf. By Lemma~\ref{lem:kernel_cokernel_in_cone} we know that $\coker(\phi) \in \langle F \oplus G \rangle$. By Lemma~\ref{lem:torsion_plus_bundle_generates_everything} a triangulated subcategory of $\Dbcoh(C)$ which contains a non-zero vector bundle and a non-zero torsion sheaf coincides with $\Dbcoh(C)$, so we have $\langle F \oplus G \rangle = \Dbcoh(C)$. This is the base case of the induction.

  For the inductive step, consider the pair of coherent sheaves $F^\prime := \ker(\phi)$ and $G^\prime := \coker(\phi)$. By Lemma~\ref{lem:kernel_cokernel_in_cone} we know that both $F^\prime$ and $G^\prime$ lie in the subcategory $\langle F \oplus G \rangle$. The sheaf $F^\prime$ is always a vector bundle since it is a subsheaf of $F$. If the sheaf $G^\prime$ is not a vector bundle, then by Lemma~\ref{lem:coherent_sheaves_on_curves} it has a non-zero torsion direct summand and thus the subcategory $\langle F \oplus G \rangle$ contains a non-zero torsion sheaf. Then, as above, by Lemma~\ref{lem:torsion_plus_bundle_generates_everything} we conclude that~$F \oplus G$ is a classical generator of~$\Dbcoh(C)$ in this case. Thus we may assume that both $F^\prime$ and $G^\prime$ are vector bundles.
  
  Since any subsheaf of $F^\prime := \ker(\phi)$ is tautologically also a subsheaf of $F$, we have the inequality~$\mu_{\max}(F^\prime) \leq \mu_{\max}(F)$. Similarly, since $G^\prime := G / \im(\phi)$ is a quotient of $G$, we also know that~$\mu_{\min}(G) \leq \mu_{\min}(G^\prime)$. Thus the pair $F^\prime, G^\prime$ satisfies the assumptions of the proposition. To apply the induction, we only have to check that the sum of ranks of~$F^\prime$ and~$G^\prime$ is strictly smaller than $\rk(F) + \rk(G)$. But, indeed, we have
  \[
    \rk(F^\prime) + \rk(G^\prime) = \rk(\ker(\phi)) + \rk(\coker(\phi)) = \rk(F) + \rk(G) - 2 \rk(\im(\phi)),
  \]
  where $\im(\phi) \subset G$ is a non-zero subsheaf since by assumption $\phi$ is a non-zero morphism, and since~$G$ is torsion-free, a non-zero subsheaf $\im(\phi)$ has strictly positive rank. Thus by induction we know that the object~$F^\prime \oplus G^\prime$ is a classical generator of $\Dbcoh(C)$. Hence the triangulated subcategory $\langle F \oplus G \rangle$ which contains the object~$F^\prime \oplus G^\prime$ is also necessarily equal to $\Dbcoh(C)$, i.e., $F \oplus G$ is a classical generator of $\Dbcoh(C)$.
\end{proof}

The conclusion is especially strong for curves of genus one:

\begin{corollary}
  \label{cor:genus_one_generators}
  Let $C$ be a smooth projective curve of genus one, and let $E \in \Dbcoh(C)$ be an object. The following are equivalent:
  \begin{itemize}
  \item $E$ is not a semistable object;
  \item $E$ is a generator of $\Dbcoh(C)$;
  \item $E$ is a classical generator of $\Dbcoh(C)$.
  \end{itemize}
\end{corollary}

\begin{proof}
  By Theorem~\ref{thm:generators_in_curves_criterion} the first two conditions are equivalent. The third condition clearly implies the second. It remains to prove that any non-semistable object $E$ is a classical generator of $\Dbcoh(C)$.

  As usual, by Lemma~\ref{lem:objects_on_curves_split} we may assume that $E$ is a coherent sheaf which is not a torsion sheaf. If the non-semistable sheaf $E$ has torsion, then by Lemma~\ref{lem:torsion_plus_bundle_generates_everything} the claim is true, so we assume that $E$ is a vector bundle. A standard argument with Serre duality (included for completeness in Lemma~\ref{lem:bundles_on_genus_one} below) shows that any vector bundle on an elliptic curve is a direct sum of semistable factors. Thus we may assume that $E \iso F \oplus G$, where $F$ and $G$ are semistable vector bundles with distinct slopes $\mu(F) < \mu(G)$. Since $g = 1$, the inequality of slopes required in Proposition~\ref{prop:sufficiently_unstable_bundles} holds and thus $E \iso F \oplus G$ is a classical generator of the category~$\Dbcoh(C)$.
\end{proof}

We used the following well-known fact in the proof above:

\begin{lemma}
  \label{lem:bundles_on_genus_one}
  Let $C$ be a smooth projective curve of genus one. Any vector bundle $\mathcal{E}$ on $C$ is isomorphic to a direct sum of semistable bundles.
\end{lemma}

\begin{proof}
  We prove this by induction on the rank of $\mathcal{E}$. Any line bundle is automatically stable and there is nothing to prove for the base case of the induction. Let $\mathcal{E}^\prime \subset \mathcal{E}$ be the maximal destabilizing subbundle. Note that $\mathcal{E}^\prime$ is semistable. Consider the short exact sequence
  \[ 0 \to \mathcal{E}^\prime \monoarrow \mathcal{E} \to \mathcal{E}^{\prime \prime} \to 0. \]
  This short exact sequence splits if and only if its extension class vanishes. By Serre duality the extension class is an element of the space
    \begin{equation}
      \label{eq:serre_duality_on_genus_one}
      \Ext^1(\mathcal{E}^{\prime \prime}, \mathcal{E}^\prime) \caniso \Hom(\mathcal{E}^\prime, \mathcal{E}^{\prime \prime})^\dual.
    \end{equation}
  By induction hypothesis we know that the bundle $\mathcal{E}^{\prime \prime}$ is a direct sum of semistable bundles. Since $\mathcal{E}^\prime$ was chosen to be maximally destabilizing each semistable factor of $\mathcal{E}^{\prime \prime} \caniso \mathcal{E} / \mathcal{E}^\prime$ has slope strictly smaller than $\mu(\mathcal{E}^\prime)$. Thus the space on the right hand side of (\ref{eq:serre_duality_on_genus_one}) is zero, so in fact $\mathcal{E}$ is isomorphic to the direct sum $\mathcal{E}^\prime \oplus \mathcal{E}^{\prime \prime}$, as claimed in the statement.
\end{proof}

For higher genus curves the conclusion of Lemma~\ref{lem:bundles_on_genus_one} does not hold, so we can only deduce the following much weaker statement.

\begin{corollary}
  \label{cor:sufficient_instability_criterion}
  Let $C$ be a smooth projective curve of genus $g$. Let $\mathcal{E}$ be a vector bundle on~$C$ such that in the sequence of slopes $\mu_1 > \mu_2 > \dots > \mu_n$ in the Harder--Narasimhan filtration of~$\mathcal{E}$ for some index $i \in [1, n-1]$ we have a large gap $\mu_i - \mu_{i + 1} > 2 g - 2$ between the slopes. Then $\mathcal{E}$ is a classical generator of $\Dbcoh(C)$.
\end{corollary}

\begin{proof}
  The condition on the slopes in the Harder--Narasimhan filtration implies that there exists a short exact sequence
  \[ 0 \to \mathcal{E}_i \to \mathcal{E} \to Q_{i} \to 0, \]
  where $\mu_{\min}(\mathcal{E}_i) = \mu_i$ and $\mu_{\max}(Q_{i}) = \mu_{i+1} < \mu_i - (2 g - 2)$. Any such sequence splits: indeed, the extension class is an element of the vector space
  \[ \Ext^1(Q_i, \mathcal{E}_i) \caniso \Hom(\mathcal{E}_i, Q_i \otimes K_C)^\dual \caniso H^0(C, (\mathcal{E}_i^\dual \otimes Q_i) \otimes K_C)^\dual, \]
  which is zero since the Harder--Narasimhan filtration on the product $\mathcal{E}_i^\dual \otimes Q_i$ is induced from the filtration on $\mathcal{E}$ and, in particular, involves only semistable factors of slopes at most 
  \[ \mu_{\max}(Q_i) - \mu_{\min}(\mathcal{E}_i) < - (2 g - 2), \]
  so even after the twist by $K_C$ all slopes are still strictly negative and hence the vector bundle~$(\mathcal{E}_i^\dual \otimes Q_i) \otimes K_C$ has no global sections.

  Thus we conclude that $\mathcal{E}$ is isomorphic to the direct sum $\mathcal{E}_i \oplus Q_i$, and the condition 
  \[ \mu_{\max}(Q_{i}) + (2 g - 2) < \mu_{\min}(\mathcal{E}_i) \]
  is in fact twice as strong as what we need to apply Proposition~\ref{prop:sufficiently_unstable_bundles}.
\end{proof}

The slope inequality in Corollary~\ref{cor:sufficient_instability_criterion} is much stronger than the one in Proposition~\ref{prop:sufficiently_unstable_bundles}. I don't know whether there is a more subtle condition:

\begin{question}
  What are some sufficient conditions for a vector bundle $\mathcal{E}$ to be a classical generator of $\Dbcoh(C)$ if $\mathcal{E}$ does not satisfy the assumption of Corollary~\ref{cor:sufficient_instability_criterion}?
\end{question}

\subsection{Generators and classical generators in families}
\label{ssec:behavior_in_families}

Semistability of a coherent sheaf is an open condition in families. Since on curves any object in the derived category splits (Lemma~\ref{lem:objects_on_curves_split}), the condition that an object in the bounded derived category of a curve is semistable is also open. Hence Theorem~\ref{thm:generators_in_curves_criterion} implies that the property of being a generator is a closed property in the sense that in any family of objects in the derived category the complement to the locus of generators is open. That is not surprising since there is a much simpler way to prove this, in fact for varieties of any dimension:

\begin{lemma}
  \label{lem:generators_are_closed}
  Let $X$ be a smooth projective variety, and let~$U$ be a smooth variety. Consider an object~$\mathcal{E} \in \Dbcoh(X \times U)$ on the product of $X$ and $U$. When considered as a family of objects in~$\Dbcoh(X)$ parametrized by~$U$, the locus of points $u \in U$ such that the derived restriction~$\mathcal{E}|_{X \times \{ u \} }$ is a generator of $\Dbcoh(X)$ is closed.
\end{lemma}

\begin{remark}
  We require $U$ to be smooth exclusively to stay in the realm of $\Dbcoh$ for exposition purposes; in general, the same conclusion holds for any perfect object $\mathcal{E}$ on the product variety~$X \times U$.
\end{remark}

\begin{proof}
  Let $u \in U$ be a point such that $\mathcal{E}|_{X \times \{ u \}}$ is not a generator of $\Dbcoh(X)$. By definition this means that there exists a non-zero object $F \in \Dbcoh(X)$ such that $\Ext^{\bullet}_{X}(\mathcal{E}|_{X \times \{ u \}}, F) = 0$. Since $X$ is proper, for any $i \in \ZZ$ the dimension of the $\Ext^{i}$-space between the restriction of $\mathcal{E}$ to a fiber and the object $F$ is an upper-semicontinuous function. Since $X$ is smooth and $\mathcal{E}$ lies in the bounded derived category, only finitely many $\Ext$ spaces are ever non-zero. Thus the vanishing of $\Ext^{\bullet}_{X}(-, F)$ holds not only for $\mathcal{E}|_{X \times \{ u \}}$, but for the restriction of $\mathcal{E}$ to the fiber over any $u^\prime \in U$ in the intersection of finitely many Zariski-open neighborhoods of~$u \in U$. In particular, $\mathcal{E}|_{X \times \{ u^\prime \}}$ is not a generator of $\Dbcoh(X)$ for all $u^\prime$ in some non-empty Zariski-open neighborhood of $u \in U$.
\end{proof}

The example in Theorem~\ref{thm:main_example} suggests that perhaps the property of an object to be a classical generator in $\Dbcoh(C)$ is also closed: if we consider the family of objects $\OO_C \oplus \mathcal{L} \in \Dbcoh(C)$ parametrized by the scheme of line bundles $\mathcal{L} \in \Pic^1(C)$ of degree one, then by Theorem~\ref{thm:main_example} and Lemma~\ref{lem:torsion_plus_bundle_generates_everything} we know that an object $\OO_C \oplus \mathcal{L}$ is a classical generator if and only if the space~$H^0(C, \mathcal{L}) \neq 0$, which is a closed condition. Similarly, in Proposition~\ref{prop:sufficiently_unstable_bundles} we prove that a direct sum $F \oplus G$ of two bundles (satisfying some conditions) is a classical generator of the category~$\Dbcoh(C)$, while a non-split extension of $F$ by $G$ may not be a classical generator (see, e.g., Remark~\ref{rem:simple_is_not_generator}); the condition that an object splits into a direct sum is closed. However, I don't know whether it is a closed condition in general:

\begin{question}
  Let $C$ be a smooth projective curve, and let $U$ be a smooth scheme. Consider an object $\mathcal{E} \in \Dbcoh(C \times U)$. Is the locus of the points $u \in U$ such that $\mathcal{E}|_{C \times \{u\}} \in \Dbcoh(C)$ is a classical generator Zariski-closed?
\end{question}

I don't know what happens for higher-dimensional varieties, but I suspect that in that case the answer is negative. Note that Theorem~\ref{thm:ravenel_neeman} shows that not being a classical generator is equivalent to a certain vanishing of $\Ext$-groups, but in a large category, where those spaces can be infinite-dimensional and the semicontinuity argument from Lemma~\ref{lem:generators_are_closed} is not available.

\subsection{Very slow classical generators}
\label{ssec:slow_generators}

Some classical generators are fast, some are slow. Bondal and Van den Bergh introduced the following definition \cite[Sec.~2.2]{bondal-vdbergh}:

\begin{definition}
  Let $T$ be a triangulated category, and let $G \in T$ be an object. For each integer $i \geq 1$ we define a subset of objects $\langle G \rangle_{i} \subset T$ via the following inductive procedure:
  \begin{itemize}
  \item $\langle G \rangle_{1} \subset T$ is the smallest set of objects containing $G$ which is closed under shifts, taking finite direct sums, and taking direct summands.
  \item $\langle G \rangle_{i} \subset T$ is defined as the set of objects $E \in T$ fitting into some distinguished triangle of the form
  \end{itemize}
  \[ E_{1} \to E \to E_{i-1} \to E_{1}[1] \]
  where $E_1 \in \langle G \rangle_{1}$ and $E_{i-1} \in \langle G \rangle_{i-1}$, as well as all direct summands of such objects $E$.
\end{definition}

Informally, $\langle G \rangle_{i}$ is the set of those objects in $T$ that can be built starting from $G$ using direct sums, shifts, direct summands, and at most $i-1$ operations of taking a cone of a map. The union $\cup_{i \geq 1} \langle G \rangle_{i}$ is equal to the triangulated subcategory $\langle G \rangle \subset T$. A classical generator~$G \in T$ is said to be \textit{strong} if for some $i \geq 1$ we have $T = \langle G \rangle_{i}$, i.e., when any object of $T$ can be constructed from $G$ using a uniformly bounded number of cones. For any object~$G \in T$ we define its \textit{generating time} $\Theta(G)$ to be the smallest $i \geq 1$ such that~$\langle G \rangle_{i+1} = T$ if $G$ is a strong generator, or infinity otherwise.

For any smooth variety $X$ the category $\Dbcoh(X)$ has a strong generator (\cite[Thm.~3.1.4]{bondal-vdbergh}). It is easy to see that if the triangulated category $T$ admits a strong generator then any classical generator of $T$ is strong. However, the generating time is very hard to control. It is expected that for a smooth projective variety $X$ the fastest generator in $\Dbcoh(X)$ has generating time equal to $\dim X$ (see \cite[Rem.~7.44]{rouquier-dimension} and \cite[Conj.~10]{orlov-curves}). This is known for curves~\cite[Thm.~6]{orlov-curves} and for a small number of specific varieties $X$ where the category $\Dbcoh(X)$ admits a nice algebraic description; in general we only know the existence of $G \in \Dbcoh(X)$ with $\Theta(G) \leq 2 \dim X$.

It is perhaps less natural, but still meaningful, to wonder whether in a fixed triangulated category strong generators can be arbitrarily slow instead of looking only at the fastest ones. In the paper \cite{ballard-favero-katzarkov} the supremum of generating times over all strong generators is bestowed with the imposing title \textit{ultimate dimension}. Whether arbitrarily slow strong generators exist is not known even for the triangulated category $\Dbcoh(C)$ with $C$ a smooth projective curve of genus at least two. Let us investigate the generating time of some classical generators of the category~$\Dbcoh(C)$ considered in this note. To do this, we start with a simple lemma, whose proof is an exercise in the octahedral axiom of triangulated categories.

\begin{lemma}
  \label{lem:multiplicative_generating_time}
  Let $T$ be a triangulated category, and let $E, F, G \in T$ be three objects. If for some $a, b \in \ZZ$ we have~$E \in \langle F \rangle_{a}$ and $F \in \langle G \rangle_{b}$, then $E \in \langle G \rangle_{a \cdot b}$.
\end{lemma}

This lemma has an obvious consequence for generating times:

\begin{corollary}
  \label{cor:generating_time_bound}
  Let $T$ be a triangulated category, and let $G \in T$ be an object. For any object~$F \in \langle G \rangle_{b}$ we have~$\Theta(G) \leq b \cdot (\Theta(F) + 1) - 1$.
\end{corollary}

\begin{proof}
  Let $a := \Theta(F) + 1$. By definition any object $E \in T$ lies in the subcategory $\langle F \rangle_{a}$. Then by Lemma~\ref{lem:multiplicative_generating_time} we know that $E$ also lies in the subcategory $\langle G \rangle_{a \cdot b}$.
\end{proof}

To start with, we find an upper bound on generating time for the most basic classical generators. This is only a very rough bound, much stronger inequalities are likely possible (in fact, \cite[Rem.~6.17]{ballard-favero-katzarkov} suggests an upper bound of $4 g$). For us, the whole point of the lemma below is that the bound depends only on the genus of the curve, not on the point $p$.

\begin{lemma}
  \label{lem:line_plus_skyscraper_generating_time}
  Let $C$ be a smooth projective curve of genus $g$, let $p \in C$ be a point, and let $\OO_p$ be the skyscraper sheaf at $p$. Then $\Theta(\OO_C \oplus \OO_p) \leq 48 g + 1$.
\end{lemma}

\begin{proof}
  Orlov proved \cite[Thm.~6]{orlov-curves} that for any line bundle $\mathcal{L}$ on $C$ of degree at least $8 g$ the generating time of the direct sum
  \[ \mathcal{L}^\dual \oplus \OO_C \oplus \mathcal{L} \oplus \mathcal{L}^2 \]
  is equal to one. Since twists by powers of $\mathcal{L}$ are autoequivalences of $\Dbcoh(C)$, they don't change the generating time. In particular, in terms of the point $p \in C$ from the statement we have
    \begin{equation}
      \label{eq:fast_curve_generator}
      \Theta(\OO_C \oplus \OO_{C}(-8 g \cdot p) \oplus \OO_{C}(-16 g \cdot p) \oplus \OO_{C}(-24 g \cdot p)) = 1.
    \end{equation}
  Using the short exact sequences
  \[ 0 \to \OO_{C}(-(m+1) \cdot p) \to \OO_{C}(-m \cdot p) \to \OO_{p} \to 0 \]
  for $m = 0, 1, \dots$ it is easy to check, inductively, that the line bundle $\OO_{C}(-m \cdot p)$ lies in the subcategory $\langle \OO_C \oplus \OO_p \rangle_{m + 1}$. Thus the generator from (\ref{eq:fast_curve_generator}) with generating time $1$ lies in the subcategory $\langle \OO_C \oplus \OO_p \rangle_{24 g + 1}$. By Corollary~\ref{cor:generating_time_bound} we conclude that
  \[ \Theta(\OO_C \oplus \OO_p) \leq (24 g + 1) (1 + 1) - 1 = 48 g + 1, \]
  as claimed in the statement.
\end{proof}

In the proof of Lemma~\ref{lem:torsion_plus_bundle_generates_everything} we have used the fact that $\OO_C \oplus \OO_p$ is a classical generator to prove that any non-torsion non-locally free object is a classical generator. The proof used the trace-splitting trick to avoid dealing with an arbitrary vector bundle, and this in fact lets us control the generating time:

\begin{lemma}
  \label{lem:torsion_plus_bundle_generating_time}
  Let $E \in \Dbcoh(C)$. If $E$ is neither a torsion object nor a locally free object, then~$E$ is a strong generator of $\Dbcoh(C)$ with~$\Theta(E) \leq 96 g + 3$.
\end{lemma}

\begin{proof}
  We follow the argument from Lemma~\ref{lem:torsion_plus_bundle_generates_everything}. We assume that $E \iso T \oplus \mathcal{E}$ with $T$ a torsion sheaf and $\mathcal{E}$ a vector bundle. The proof of Lemma~\ref{lem:torsion_plus_bundle_generates_everything} shows that for any point $p \in \supp(T)$ the skyscraper sheaf $\OO_p$ lies in the subset $\langle T \rangle_{2}$. Let $F \in \Dbcoh(C)$ be an arbitrary object. Then by Lemma~\ref{lem:line_plus_skyscraper_generating_time} we know that
  \[ \mathcal{E}^\dual \otimes F \in \langle \OO_C \oplus \OO_p \rangle_{48 g + 2}, \]
  and the proof of Lemma~\ref{lem:torsion_plus_bundle_generates_everything} shows that $F$ is a direct summand of the object
  \[ \mathcal{E} \otimes (\mathcal{E}^\dual \otimes F) \in \langle \mathcal{E} \oplus \OO_p \rangle_{48 g + 2}, \]
  and hence by Lemma~\ref{lem:multiplicative_generating_time} we have
  \[ F \in \langle \langle \mathcal{E} \oplus T \rangle_{2}  \rangle_{48 g + 2} \subset \langle \mathcal{E} \oplus T \rangle_{96 g + 4}, \]
  which means that the generating time of $\mathcal{E} \oplus T$ is at most $96 g + 3$.
\end{proof}

Unlike Lemma~\ref{lem:torsion_plus_bundle_generates_everything}, the generating time of classical generators constructed in Proposition~\ref{prop:sufficiently_unstable_bundles} is hard to bound in terms of the genus of the curve since the proof uses induction on the rank of vector bundles, which could be arbitrarily high. However, if we require a stronger condition on slopes, the inductive procedure could be performed in a single step, and we get an upper bound on the generating time.

\begin{lemma}
  \label{lem:sufficiently_unstable_generating_time}
  Let $C$ be a smooth projective curve of genus $g$. Let $F$ and $G$ be two vector bundles on $C$ such that $\mu_{\max}(F) + 2 g < \mu_{\min}(G)$. Then $F \oplus G$ is a classical generator of the category~$\Dbcoh(C)$ with $\Theta(F \oplus G) \leq 192 g + 7$.
\end{lemma}

\begin{proof}
  Pick a point $p \in C$. Consider the vector bundle $\mathcal{E} := F^\dual \otimes G \otimes \OO_{C}(- p)$. By the inequality of the slopes each semistable factor in the Harder--Narasimhan filtration of $\mathcal{E}$ has slope strictly larger than $2 g - 1$. It is well-known that any semistable bundle of slope strictly larger than $2 g - 1$ is globally generated and has no higher cohomology (see Lemma~\ref{lem:effective_global_generation_bound} below); this implies that $\mathcal{E}$, being an iterated extension of such bundles, is also globally generated. Let $q \in C$ be a point not equal to $p$. Since the vector bundle $\mathcal{E} \caniso F^\dual \otimes G \otimes \OO_{C}(- p)$ is globally generated, any map
  \[ F|_q \to (G \otimes \OO_{C}(- p))|_q \]
  of the fibers at the point $q$ can be lifted to a globally defined morphism $F \to G \otimes \OO_{C}(- p)$. For some $r \in \ZZ$ there exists a surjection
  \[ (F^{\oplus r})|_q \to (G \otimes \OO_{C}(- p))|_q, \]
  which therefore lifts to some morphism $\phi\colon F^{\oplus r} \to G \otimes \OO_{C}(- p)$ which is surjective at the point $q$. Consider the composition
  \[ F^{\oplus r} \xrightarrow{\phi} G \otimes \OO_{C}(- p) \xrightarrow{\cdot p} G. \]
  This is a morphism of vector bundles which on the fiber at $q$ is surjective and on the fiber at $p$ is zero. Thus the cokernel of this composition is a non-trivial torsion sheaf $T$. By Lemma~\ref{lem:kernel_cokernel_in_cone} we obtain that $T \in \langle F \oplus G \rangle_{2}$. Hence by Lemma~\ref{lem:torsion_plus_bundle_generating_time} we get a bound
  \[ \Dbcoh(C) = \langle T \oplus F \rangle_{96 g + 4} \subset \langle F \oplus G \rangle_{192 g + 8}, \]
  so the generating time of $F \oplus G$ is at most $192 g + 7$.
\end{proof}

\begin{remark}
  Of course, this is a very lazy upper bound; for example, it is clear that the torsion sheaf~$T$ constructed in the proof already contains a skyscraper sheaf~$\OO_p$ as a direct summand, and appealing to Lemma~\ref{lem:torsion_plus_bundle_generating_time} wastes half of the generating time to construct a skyscraper from itself.
\end{remark}

In the proof we used the following well-known fact (see, e.g.,~\cite[Lem.~3.2.6]{spec-book}). It is an exercise on Serre duality and we leave the proof as an exercise to the reader.

\begin{lemma}
  \label{lem:effective_global_generation_bound}
  Let $C$ be a smooth projective curve of genus $g$, and let $E$ be a semistable vector bundle on $C$ with slope $\mu(E) > 2 g - 1$. Then $E$ is globally generated and $H^1(C, E) = 0$.
\end{lemma}

The proof of Lemma~\ref{lem:sufficiently_unstable_generating_time} suggests that the number of truly necessary inductive steps in the proof of Proposition~\ref{prop:sufficiently_unstable_bundles} could possibly be independent of the ranks of $F$ and $G$. I don't know whether that is true and it feels like a difficult problem. Thus, the question:

\begin{question}
  Is there an upper bound for the generating times of classical generators constructed in Proposition~\ref{prop:sufficiently_unstable_bundles} in terms of the genus $g$ of the curve?
\end{question}

I am slightly more inclined to say no. In fact, I expect that there are arbitrarily slow classical generators of $\Dbcoh(C)$ in the gap between Proposition~\ref{prop:sufficiently_unstable_bundles} and Lemma~\ref{lem:sufficiently_unstable_generating_time}, but this is just a feeling, not a mathematical argument. Note, however, that for curves of genus one it is known that any classical generator has generating time at most 4, see \cite{ballard-favero-katzarkov}; in some sense, the reason is that Lemma~\ref{lem:bundles_on_genus_one} allows for a quick reduction to the generator which behaves like $\OO_C \oplus \OO_p$.

\printbibliography
\end{document}